\documentclass[11pt]{amsart}
\usepackage{amsmath, amsthm, amssymb, amscd, graphicx, times, hyperref, epstopdf,subfig}
\usepackage[hmargin=1.2in,vmargin=1.2in]{geometry}
\newtheorem{thm}{Theorem}[section]
\newtheorem{cor}[thm]{Corollary}

\newtheorem{lem}[thm]{Lemma}

\theoremstyle{definition}
\newtheorem{defn}[thm]{Definition}

\theoremstyle{remark}

\numberwithin{equation}{section}
\newtheoremstyle{dotless}{}{}{}{}{\bfseries}{}{ }{}
\theoremstyle{dotless}
\newtheorem*{repthm}{Theorem}
\newtheorem*{repcor}{Corollary}

\newcommand{\co}{\colon\thinspace}

\def\R{{\mathbb R}}

\newcommand{\ra}{\rightarrow}

\begin{document}
\title {Explicit Equivalences between CAT(0) hyperbolic type geodesics}
\author{Harold Sultan}
\address{Department of Mathematics\\Brandeis University\\
Waltham\\MA 02453}
\email{HSultan@Brandeis.edu}
\date{\today}

\begin{abstract}
We prove an explicit equivalence between various hyperbolic type properties for quasi-geodesics in CAT(0) spaces.  Specifically, we prove that for $X$ a CAT(0) space and $\gamma \subset X$ a quasi-geodesic, the following four statements are equivalent and moreover the quantifiers in the equivalences are explicit:  (i) $\gamma$ is S-Slim, (ii) $\gamma$ is M(K,L)--Morse, (iii) $\gamma$ is (b,c)--contracting, and (iv) $\gamma$ is C--strongly contracting.  In particular, this explicit equivalence proves that for $f$ a $(K,L)$ quasi-isometry between CAT(0) spaces, and $\gamma$ a C--strongly contracting (K',L')--quasi-geodesic, then $f(\gamma)$ is a $C'(C,K,L,K',L')$--strongly contracting quasi-geodesic.  This result is necessary for a key technical point with regard to Charney's contracting boundary for CAT(0) spaces.
\end{abstract}
\maketitle

\tableofcontents

\section{Introduction and Overview} \label{sec:intro}
In the study of spaces of non-positive curvature, Euclidean and hyperbolic space represent the two classically well understood extreme ends of the spectrum.  More generally, in the literature a robust approach for studying spaces of interest is to identify particular directions, geodesics, or subspaces of the space in question which share features in common with one of these two prototypes.  In particular, with regard to identifying \emph{hyperbolic type geodesics} in spaces of interest, or geodesics which share features in common with geodesics in hyperbolic space, there are various well studied precise notions including being Morse, being contracting, and being slim.  Specifically, such studies have proven fruitful in analyzing right angled Artin groups \cite{behrstockcharney}, Teichm\"uller space \cite{behrstock,brockfarb,brockmasur,bmm,mosher}, the mapping class group \cite{behrstock}, CAT(0) spaces \cite{sultanmorse,behrstockdrutu,bestvina,charney}, and Out($F_{n}$) \cite{algomkfir} amongst others (See for instance\cite{drutumozessapir,drutusapir,kapovitchleeb,osin,mm1}). 

A Morse geodesic $\gamma$ is defined by the property that all quasi-geodesics $\sigma$ with endpoints on $\gamma$ remain within a bounded distance from $\gamma.$  A  strongly contracting geodesic has the property that metric balls disjoint from the geodesic have nearest point projections onto the geodesic with uniformly bounded diameter.  A geodesic is called slim if geodesic triangles with one edge along the geodesic are $\delta$-thin.  It is an elementary fact that in hyperbolic space, or more generally $\delta$-hyperbolic spaces, all quasi-geodesics are Morse, strongly contracting, and slim.  On the other hand, in product spaces such as Euclidean spaces of dimension two and above, there are no Morse, strongly contracting, or slim quasi-geodesics.

Building on results in \cite{sultanmorse}, in this paper we prove that the various aforementioned hyperbolic type properties are equivalent and moreover the quantifiers in the equivalences are explicit. 
\begin{repthm}$\textbf{\ref{thm:main}.}$(Main Theorem). \emph{
Let $X$ be a CAT(0) space and $\gamma \subset X$ a quasi-geodesic.  Then the following are equivalent:
\begin{enumerate}
\item $\gamma$ is (b,c)--contracting,
\item $\gamma$ is $C'$--strongly contracting,
\item $\gamma$ is $M$--Morse, and
\item $\gamma$ is $S$--slim
\end{enumerate}
Moreover, any one of the four sets of constants $\{(b,c),$C'$,M,S\}$ can be written in terms of any of the others.}
\end{repthm}

Theorem \ref{thm:main} should be considered in the context of related theorems in \cite{bestvina,behrstock,charney,drutumozessapir,kapovitchleeb,sultanmorse} among others.  In particular, in \cite{bestvina} geodesics with property (2) are studied and in fact among other things it is shown that for the case of $\gamma$ a geodesic (2)$\implies$(4).  In \cite{charney} geodesics with property (2) are studied and it is shown that (2)$\implies$(3), an explicit proof of which also appears in \cite{algomkfir}.  In \cite{drutumozessapir} geodesics with property (3) are studied.  In \cite{sultanmorse} building on work of the previous authors it is shown that properties (1),(2), and (3) are equivalent, although the proof relies on limiting arguments and hence the constants of the equivalence could not be recovered.

As a corollary of Theorem \ref{thm:main} we highlight the following consequence, which in fact served as motivation for the results in this paper. 
\begin{repcor} $\textbf{\ref{cor:quasiisometry}.}$ 
\emph{Let $X$ be a CAT(0) space, $\gamma \subset X$ a C--strongly contracting (K',L')--quasi-geodesic, and $f:X \ra X$ a $(K,L)$ quasi-isometry.  Then $f(\gamma)$ is $C'(C,K,L,K',L')$--strongly contracting quasi-geodesic. }
\end{repcor}
In particular, Corollary \ref{cor:quasiisometry} is very useful in \cite{charney} where it is used to show that self quasi-isometries of CAT(0) spaces give rise to continuous maps on Charney's contracting boundary for CAT(0) spaces.

 \subsection*{Acknowledgements}
$\\$
$\indent$
I want to thank Ruth Charney for motivating conversations as well as Michael Carr and Thomas Koberda for useful conversations and insights regarding arguments and ideas in this paper.

\section{Background}\label{sec:back}
\subsection{Quasi-geodesics and CAT(0) spaces}
\begin{defn}[quasi-geodesic] \label{defn:quasi} A \emph{(K,L) quasi-geodesic} $\gamma \subset X$ is the image of a map $\gamma:I \ra X$ where $I$ is a connected interval in $\R$ (possibly all of $\R$) such that $\forall s,t \in I$ we have the following quasi-isometric inequality:
\begin{equation} \label{eq:qi} \frac{|s-t|}{K}-L \leq  d_{X}(\gamma(s),\gamma(t)) \leq K|s-t|+L
\end{equation} 
We refer to the quasi-geodesic $\gamma(I)$ by $\gamma,$ and when the constants $(K,L)$ are not relevant omit them.   
\end{defn}


CAT(0) spaces are geodesic metric spaces defined by the property that triangles are no ``fatter'' than the corresponding comparison triangles in Euclidean space.  In particular, using this property one can prove the following lemma, see \cite[Section II.2]{bridson} for details. 

\begin{lem} Let $X$ be a CAT(0) space.
\begin{enumerate} \label{lem:cat}
\item [C1:] (Projections onto convex subsets). Let $C$ be a convex subset, complete in the induced metric, then there is a well-defined distance non-increasing nearest point projection map $\pi_{C}\co X \ra C.$  In particular, $\pi_{C}$ is continuous.  We will often consider the case where $C$ is a geodesic. 
\item [C2:] (Convexity). Let $c_{1}\co [0,1] \ra X$ and $c_{2}\co [0,1] \ra X$ be any pair of geodesics parameterized proportional to arc length.  Then the following inequality holds for all $t\in [0,1]:$
$$d(c_{1}(t),c_{2}(t)) \leq (1-t) d(c_{1}(0),c_{2}(0)) + td(c_{1}(1),c_{2}(1))  $$
\end{enumerate}
\end{lem}  
\subsection{Hyperbolic type quasi-geodesics}  In this section we define the hyperbolic types of quasi-geodesics we will consider in this paper.  The following definition of Morse (quasi-)geodesics has roots in the classical paper \cite{morse}:
\begin{defn}[Morse quasi-geodesics] A (quasi-)geodesic $\gamma$ is called an M--\emph{Morse (quasi-)geodesic} if for every $(K,L)$-quasi-geodesic $\sigma$ with endpoints on $\gamma,$ we have $\sigma \subset N_{M(K,L)}(\gamma).$   That is, $\sigma$ is within a bounded distance, $M=M(K,L),$ from $\gamma,$ with the bound depending only on the constants $K,L.$  In the literature, Morse (quasi-)geodesics are sometimes referred to as \emph{stable quasi-geodesics}.
\end{defn} 

The following generalized notion of contracting quasi-geodesics can be found for example in \cite{behrstock,brockmasur}, and is based on a slightly more general notion of (a,b,c)--contraction found in \cite{mm1} where it serves as a key ingredient in the proof of the hyperbolicity of the curve complex.
\begin{defn}[contracting quasi-geodesics] A (quasi-)geodesic $\gamma$ is said to be \emph{(b,c)--contracting} if $\exists$ constants $0<b\leq 1$ and $0<c$ such that $\forall x,y \in X,$ $$d_{X}(x,y)<bd_{X}(x,\pi_{\gamma}(x))  \implies d_{X}(\pi_{\gamma}(x),\pi_{\gamma}(y))<c.$$
For the special case of a (b,c)--contracting quasi-geodesic where $b$ can be chosen to be $1,$ the quasi-geodesic $\gamma$ is called \emph{c--strongly contracting.}
\end{defn}  

The following elementary lemma shows that given a $(b,c)$--contracting quasi-geodesic one can increase $b$ to be arbitrarily close to $1$ at the expense of increasing $c.$
\begin{lem} \label{lem:close}
If $\gamma$ is (b,c)--contracting quasi-geodesic, then for any arbitrarily small $\epsilon>0,$ the quasi-geodesic $\gamma$ is ($1-\epsilon,c'(\epsilon,b,c))$--contracting. 
\end{lem}

\begin{proof}
Notice that if $\gamma$ is $(b,c)$--contracting, then it is also $(b+b(1-b),2c)$--contracting.  Similarly, it is also $(b+b(1-b)+b(1-b)^{2},3c)$--contracting.  Iterating this process, the statement of the lemma follows, as for $0<b<1$ the sum of the geometric series $\sum_{i=0}^{\infty}b(1-b)^{i}$ converges to $1.$  
\end{proof}

Finally, the following definition of a slim quasi-geodesic is introduced in \cite{bestvina}.
\begin{defn}[slim quasi-geodesics] \label{def:slim}A (quasi-)geodesic $\gamma$ is said to be \emph{S--slim} if $\exists$ constant $S$ such that for all $x \in X$ and $y \in \gamma,$ we have:
$$d(\pi_{\gamma}(x),[x,y]) \leq S.$$
\end{defn}

Note that if $\gamma$ is an S--slim quasi-geodesic, then $$ |[x,\pi_{\gamma}(x)]| + |[\pi_{\gamma}(x),y]|-2S \leq |[x,y]| \leq  |[x,\pi_{\gamma}(x)]| + |[\pi_{\gamma}(x),y]|.$$   Moreover, if $z \in [x,y]$ is a point such that $d(y,z)=d(y,\pi_{\gamma}(x))$ (or similarly such that $d(x,z)=d(x,\pi_{\gamma}(x))$), then $d(z,\pi_{\gamma}(x))\leq 2S.$  

We conclude this section by citing a lemma relating contracting and slim geodesics.
\begin{lem}[\cite{bestvina} Lemma 3.5] \label{lem:bestvina} Let $\gamma$ be a C--strongly contracting geodesic in a CAT(0) space.  Then $\gamma$ is $(3C+1)$-slim. 
\end{lem}

\section{Main Theorem and Proof}
Throughout this section we will assume we are in the setting of a CAT(0) metric space X.  The following two elementary lemmas regarding the concatenation of geodesic segments will be useful in the proof of Theorem \ref{thm:main}.
\begin{lem}\label{lem:quasigeodesic}
For any triple of points $a,b,c \in X,$ the concatenated path $$\phi=[a,\pi_{[b,c]}(a)] \cup [\pi_{[b,c]}(a),c],$$ is a (3,0) quasi-geodesic.  
\end{lem}
\begin{proof}
We must show that $\forall x,y \in \phi,$ the (3,0)--quasi-isometric inequality of Equation \ref{eq:qi} is satisfied.  Since $\phi$ is a concatenation of two geodesic segments, without loss of generality we can assume $x \in  [a,\pi_{[b,c]}(a)], y \in [\pi_{[b,c]}(a),c].$  Since $x \in [a,\pi_{[b,c]}(a)]$ it follows that $\pi_{[b,c]}(x)=\pi_{[b,c]}(a),$ and hence $d(x,\pi_{[b,c]}(a))\leq d(x,y).$  Let $d_{\phi}(x,y)$ denote the distance along $\phi$ between $x$ and $y.$  

Then, the following inequality completes the proof:
\begin{eqnarray*}
d(x,y) \leq d_{\phi}(x,y) &=& d(x,\pi_{[b,c]}(a)) + d(\pi_{[b,c]}(a),y) \\
& \leq& d(x,\pi_{[b,c]}(a)) + \left(d(\pi_{[b,c]}(a),x) + d(x,y) \right) \\ 
 &\leq& 2d(x,\pi_{[b,c]}(a)) + d(x,y)  \leq 3d(x,y)
\end{eqnarray*} 
\end{proof}

Building on Lemma \ref{lem:quasigeodesic}, presently we will prove a lemma which ensures that the concatenation of five geodesic segments under certain hypothesis is a quasi-geodesic with controlled quasi-constants.

Let $\gamma$ be a geodesic, and $x,y \in X.$  Set $D=d(\pi_{\gamma}(x),\pi_{\gamma}(y)).$  Let $a,b,c$ be constants such $d(x,\pi_{\gamma}(x)) = aD, d(x,y) = bD, d(y,\pi_{\gamma}(y)) = cD.$  See Figure \ref{fig:cases}.  Note that by property [C1] of Lemma \ref{lem:cat}, $b \geq 1.$  Consider the continuous function $\rho_{1}(z)=d([x,\pi_{\gamma}(x)],z).$  If we restrict the function $\rho_{1}$ to the geodesic $[x,y],$ by definition $\rho_{1}(x)=0,\rho_{1}(y) \geq D,$ where the latter inequality follows from property [C1] of Lemma \ref{lem:cat}.  Then by intermediate value theorem there is some point $s \in [x,y]$ such that $\rho_{1}(s)=\frac{D}{4}.$  Moreover, we can assume $s \in [x,y]$ is the point in $[x,y]$ closest to $y$ such that $\rho_{1}(s)=\frac{D}{4}.$  Similarly, we can define the continuous function $\rho_{2}(z)=d([y,\pi_{\gamma}(y)],z).$  If we restrict the function $\rho_{2}$ to the geodesic $[x,y],$ then as above the intermediate value theorem ensures there is some point $t \in [x,y]$ such that $\rho_{2}(t)=\frac{D}{4},$ and moreover we can assume $t \in [x,y]$ is the point in $[x,y]$ closest to $x$ such that $\rho_{2}(t)=\frac{D}{4}.$  Notice that since $d(s,[x,\pi_{\gamma}(x)])=\frac{D}{4},d(t,[y,\pi_{\gamma}(y)])=\frac{D}{4},$ and $d(x,y) \geq d(\pi_{\gamma}(x),\pi_{\gamma}(y)=D,$ it follows that $s$ precedes $t$ along the geodesic $[x,y].$  In fact, if follows that $d(s,t)\geq \frac{D}{2}.$  

Let $r$ be a point in $[x,\pi_{\gamma}(x)]$ such that $d(s,r)=\frac{D}{4}.$  Similarly, let $u$ be a point in $[y,\pi_{\gamma}(y)]$ such that $d(t,u)=\frac{D}{4}.$  Note that $r,u$ are uniquely defined as they are nearest point projections, that is $\pi_{[x,\pi_{\gamma}(x)]}(s)=r$ and similarly $\pi_{[y,\pi_{\gamma}(y)]}(t)=u.$  Furthermore, by construction we similarly have that $\pi_{[s,y]}(r)=s$ and $\pi_{[x,t]}(u)=t.$  

\begin{figure}[htpb] 
\centering
\includegraphics[height=5 cm]{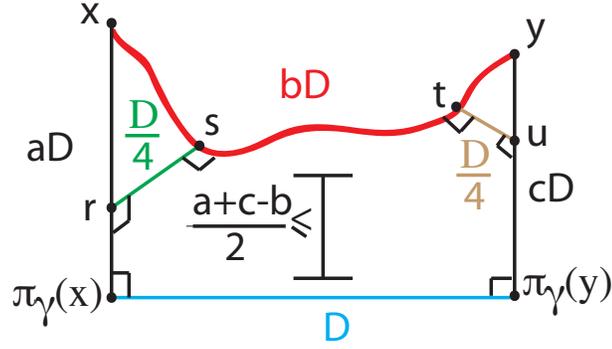}
\caption{Illustration of Lemma \ref{lem:genquasi}.}\label{fig:cases}
\end{figure}        

\begin{lem} \label{lem:genquasi}
In the situation described above, the concatenation $$\phi=[\pi_{\gamma}(x),r] \cup [r,s] \cup [s,t] \cup [t,u] \cup [u,\pi_{\gamma}(y)],$$ is a $(\left(1+4(a+b+c)\right), 0)$-quasi-geodesic.  In particular, if $a+c>b$ and $\gamma$ is M(K,L)-Morse, then $$D \leq \frac{2M(1+4(a+b+c),0)}{a+c-b}.$$
\end{lem}

\begin{proof} We will show that $\forall w,z \in \phi,$ that the $(\left(1+4(a+b+c)\right), 0)$--quasi-isometric inequality of Equation \ref{eq:qi} is satisfied.  Since $\phi$ is a concatenation of geodesics, without loss of generality we can assume $a,b$ belong to different geodesic segments within $\phi.$   Since there are 5 different geodesic segments in $\phi,$ there are ${5 \choose 2} =10,$ cases to consider.  By Lemma \ref{lem:quasigeodesic} we know that the (3,0)--quasi-isometric inequality is satisfied in the case where $w$ and $z$ belong to adjacent geodesic segments in the concatenation.  Since $b\geq 1$ it follows that $1+4(a+b+c) >3,$ and in particular, the  $(\left(1+4(a+b+c)\right), 0)$-quasi-isometric inequality is satisfied.  To complete the proof of the first statement of the lemma we will consider the six remaining cases and in each case verify the quasi-isometric inequality:

\begin{enumerate}
\item $w\in [\pi_{\gamma}(x),r], z \in [s,t]:$ By definition, in this case $\frac{D}{4}=d(r,s) \leq d(w,z).$  Hence,  
\begin{eqnarray*}
d(w,z) \leq d_{\phi}(w,z) &=& d(w,r) + |[r,s]| + d(s,z) \leq  |[\pi_{\gamma}(x),x]| + |[r,s]| + |[s,t]| \\
&\leq& aD + \frac{D}{4} + bD = \frac{D}{4}(4a+4b+1)\\
 &\leq&  d(w,z)(1+ 4a+4b)
\end{eqnarray*} 

\item $w\in [r,s], z \in [t,u]:$ By definition, in this case $\frac{D}{2} \leq d(s,t) \leq d(w,z).$  Hence,  
\begin{eqnarray*}
d(w,z) \leq d_{\phi}(w,z) &=& d(w,s) + |[s,t]| + d(t,z) \leq  |[r,s]| + |[s,t]| + |[t,u]| \\
&\leq& \frac{D}{4} + bD + \frac{D}{4} = \frac{D}{2}(1+2b)\\
 &\leq&  d(w,z)(1+2b)
\end{eqnarray*} 

\item $w\in [s,t], z \in [u,\pi_{\gamma}(y)]:$ By definition, in this case $\frac{D}{4}=d(t,u) \leq d(w,z).$  Hence,  
\begin{eqnarray*}
d(w,z) \leq d_{\phi}(w,z) &=& d(w,t) + |[t,u]| + d(u,z) \leq  |[s,t]| + |[t,u]| +  |[\pi_{\gamma}(y),y]| \\
&\leq& bd + \frac{D}{4} +cD = \frac{D}{4}(4b+1+4c)\\
 &\leq&  d(w,z)(1+4b+4c)
\end{eqnarray*} 

\item $w\in [\pi_{\gamma}(x),r], z \in [t,u]:$  By property [C1] of Lemma \ref{lem:cat} in this case,
\begin{eqnarray*}
d(w,z) &\geq& d(\pi_{\gamma}(w),\pi_{\gamma}(z))  = d(\pi_{\gamma}(x),\pi_{\gamma}(z)) \\
&\geq& d(\pi_{\gamma}(x),\pi_{\gamma}(y)) - d(\pi_{\gamma}(y),\pi_{\gamma}(z))  =D- d(\pi_{\gamma}(u),\pi_{\gamma}(z)) \\    
&\geq&  D-|[u,z]| \geq D-|[u,t]| = \frac{3D}{4}.  \\   
\end{eqnarray*}
Then, the following inequality proves the desired quasi-isometric inequality in this case:
\begin{eqnarray*}
d(w,z) \leq d_{\phi}(w,z) &=& d(w,r) + |[r,s]| +|[s,t]|+ d(t,z) \\
&\leq&  |[\pi_{\gamma}(x),x]| + |[r,s]| + |[s,t]| + |[t,u]| \\
&\leq& aD + \frac{D}{4} + bD +\frac{D}{4}= \frac{D}{4}(2+4a+4b)\\
 &\leq&  d(w,z)\frac{(2+4a+4b)}{3}
\end{eqnarray*} 

\item $w\in [r,s], z \in [u,\pi_{\gamma}(y)]:$  As in the previous case, by property [C1] of Lemma \ref{lem:cat} in this case, we have that $d(w,z) \geq \frac{3D}{4}.$  Hence,
\begin{eqnarray*}
d(w,z) \leq d_{\phi}(w,z) &=& d(w,s) +|[s,t]|+ |[t,u]|+ d(u,z) \\
&\leq&  |[r,s]| + |[s,t]| + |[t,u]| + |[\pi_{\gamma}(y),y]|  \\
&\leq&  \frac{D}{4} + bD +\frac{D}{4} +cD= \frac{D}{4}(2+4b+4c)\\
 &\leq& d(w,z) \frac{(2+4b+4bc)}{3}
\end{eqnarray*} 

\item $w\in [\pi_{\gamma}(x),r], z \in [u,\pi_{\gamma}(y)]:$ By property [C1] of Lemma \ref{lem:cat} in this case $d(w,z) \geq D.$  Hence,
\begin{eqnarray*}
d(w,z) \leq d_{\phi}(a,b) &=& d(w,r) + |[r,s]| +|[s,t]|+ |[t,u]|+ d(u,z) \\ 
&\leq&   |[\pi_{\gamma}(x),x]|+ |[r,s]| + |[s,t]| + |[t,u]| + |[\pi_{\gamma}(y),y]|  \\
&\leq&  aD+ \frac{D}{4} + bD +\frac{D}{4} +cD= D(\frac{1}{2}+a+b+c)\\
 &\leq&  d(w,z)(\frac{1}{2}+a+b+c)
 \end{eqnarray*} 
\end{enumerate}

For the ``in particular'' clause of the lemma note that if $a+c>b$ then $d([x,y],\gamma) \geq D\frac{a+c-b}{2}.$  Since $[s,t] \subset [x,y],$ in particular $d([s,t],\gamma) \geq D\frac{a+c-b}{2}.$  On the other hand, $[s,t]$ is a non-trivial portion of the quasi-geodesic $\phi$ and hence must stay within a neighborhood of $\gamma$ controlled by the Morse constant of $\gamma.$  Specifically, $[s,t] \subset N_{M(1+4(a+b+c),0)}(\gamma).$  Combining the inequalities completes the proof of the lemma.  
\end{proof}

The following lemma and ensuing corollary will be be useful in the proof of Theorem \ref{thm:main}.  Specifically, these results will be used to reduce arguments regarding quasi-geodesics to case of geodesics.  The lemma is closely related to and should be compared with Lemma 3.8 of \cite{bestvina}.
\begin{lem} \label{lem:reduce}
Let $\gamma$ be an M--Morse, C--strongly contracting geodesic, and let $\gamma'$ be a (K,L)--quasi-geodesic with endpoints on $\gamma.$  Then $\gamma'$ is $C'(C,M)$-strongly contracting.  Similarly, let $\gamma$ be an M--Morse $(b,c)$--contracting geodesic, and let $\gamma'$ be a (K,L)--quasi-geodesic with endpoints on $\gamma.$  Then $\gamma'$ is $(b,c'(c,M)$--contracting.  
\end{lem}

We will prove the first statement of the lemma.  The proof of the ``similarly'' statement is identical.

Since nearest point projections onto quasi-geodesics is not uniquely determined, in the proof of the lemma we will use the convention that $\pi_{\gamma'}(x)$ represents an arbitrary element in the nearest point projection set of $x$ onto $\gamma'.$  Additionally, when measuring distances between elements in nearest point projection sets, such as $d(\pi_{\gamma'}(x),\pi_{\gamma'}(y)),$ we will use the convention that the distance is the supremum over all possible choices of elements in the nearest point projection sets.  That is, $$d(\pi_{\gamma'}(x),\pi_{\gamma'}(y)) =: \sup\{d(x',y') | x' \in \pi_{\gamma'}(x), y' \in \pi_{\gamma'}(y)\}.$$

\begin{proof}  First we will prove that $\forall z \in X, d(\pi_{\gamma'}(z),\pi_{\gamma'}(\pi_{\gamma}(z)))$ is bounded above in terms of the constants $C,M.$   Set the Morse constant $M(K,L)=M.$  Then, 

\begin{equation}
\label{eq01}
  \begin{aligned}
   d(z,\pi_{\gamma}(\pi_{\gamma'}(z)))  &\leq& d(z,\pi_{\gamma'}(z)) + d(\pi_{\gamma'}(z),\pi_{\gamma}(\pi_{\gamma'}(z))) \leq d(z,\pi_{\gamma'}(z)) +M \\
&\leq& d(z,\pi_{\gamma}(z))  + d(\pi_{\gamma}(z),\gamma')+M \leq  d(z,\pi_{\gamma}(z))  +2M.
  \end{aligned}
\end{equation}

By Lemma \ref{lem:bestvina}, the geodesic $\gamma$ is (3C+1)-slim.  Consider the triangle $\triangle(z,\pi_{\gamma}(z),\pi_{\gamma}(\pi_{\gamma'}(z))).$   
(3C+1)--slimness in conjunction with Equation \ref{eq01} implies that 
\begin{eqnarray} \label{eq:r1}d(\pi_{\gamma}(z),\pi_{\gamma}(\pi_{\gamma'}(z))) \leq 2M +2(3C+1).\end{eqnarray}
Since $\gamma' \subset N_{M}(\gamma),$ by the triangle inequality 
\begin{eqnarray} \label{eq:r2}d(\pi_{\gamma}(z),\pi_{\gamma}(\pi_{\gamma'}(\pi_{\gamma}(z)))) \leq 2M.
\end{eqnarray}
Combining Equations \ref{eq:r1} and \ref{eq:r2}, by the triangle inequality we have 
\begin{eqnarray} \label{eq:r3}d(\pi_{\gamma}(\pi_{\gamma'}(\pi_{\gamma}(z))),\pi_{\gamma}(\pi_{\gamma'}(z))) \leq 4M +2(3C+1).
\end{eqnarray}
Finally, using the Equation \ref{eq:r3} in conjunction with the fact that $\gamma' \subset N_{M}(\gamma)$ and the triangle inequality, it follows that $\forall z\in X,$
\begin{eqnarray}\label{eq:r4}
d(\pi_{\gamma'}(\pi_{\gamma}(z)),\pi_{\gamma'}(z)) \leq 6M+2(3C+1).
\end{eqnarray}


Now assume we have $x,y \in X$ such that $d(x,y) < d(x, \pi_{\gamma'}(x)).$  We must show that we can bound $d(\pi_{\gamma'}(x),\pi_{\gamma'}(y))$ from above in terms of the constants $C,M.$  Since $d(x,y) < d(x, \pi_{\gamma'}(x)) \leq d(x, \pi_{\gamma}(x)) +M,$ using the facts that $\gamma$ is $C$--strongly contracting and nearest point projections onto geodesics are distance non-increasing, we have that
\begin{eqnarray} \label{eq:r4.5} d(\pi_{\gamma}(x), \pi_{\gamma}(y)) \leq C +M.
\end{eqnarray}
As above, using the fact that $\gamma' \subset N_{M}(\gamma),$ in conjunction with Equation \ref{eq:r4.5} and the triangle inequality, it follows that \begin{eqnarray}\label{eq:r5}
d(\pi_{\gamma'}(\pi_{\gamma}(x)),\pi_{\gamma'}(\pi_{\gamma}(y))) \leq C+ 3M.
\end{eqnarray}  Putting together Equations \ref{eq:r4} and \ref{eq:r5}, the following completes the proof:
\begin{eqnarray*}
d(\pi_{\gamma'}(x),\pi_{\gamma'}(y)) &\leq& d(\pi_{\gamma'}(x),\pi_{\gamma'}(\pi_{\gamma}(x))) + d(\pi_{\gamma'}(\pi_{\gamma}(x)),\pi_{\gamma'}(\pi_{\gamma}(y))) + d(\pi_{\gamma'}(\pi_{\gamma}(y)),\pi_{\gamma'}(y)) \\
&\leq& (6M+2(3C+1)) + (C+ 3M) +(6M+2(3C+1)) = 15M +7C +4
\end{eqnarray*}  
\end{proof}

As a corollary of Lemma \ref{lem:reduce} we have the following:
\begin{cor} \label{cor:reduce}
If it's true that a geodesic being M-Morse implies that the geodesic is $C(M)$--strongly contracting, then it's also true that a quasi-geodesic being $M'$-Morse implies that the quasi-geodesic is $C'(M')$--strongly contracting.  Similarly, if it's true that a geodesic being M-Morse implies that the geodesic is $(b,c)$--contracting, then it's also true that a quasi-geodesic being $M'$-Morse implies that the quasi-geodesic is $(b,c'(M'))$--contracting. 
\end{cor}

\begin{proof}  Once again we will prove the first statement, and the ``similarly'' statement follows identically.  Assume that if a geodesic is M--Morse then it is also $C(M)$--strongly contracting.  Let $\gamma'$ be an $M'$--Morse quasi-geodesic.  Fix $x\in X$ and $x' \in \pi_{\gamma'}(x).$  Let $y\in X$ be such that $d(x,y) < d(x,x'),$ and fix $y' \in \pi_{\gamma'}(y).$   Notice that $$d(x',y') \leq d(x',x) + d(x,y) +d(y,y') \leq  d(x',x)  + d(x',x) + 2d(x',x) = 4d(x',x).$$  

Let $\alpha_{x} \in \gamma'$ be any point preceding $x'$ such that $d(\alpha_{x},x') \geq 4d(x',x).$  Similarly, let $\beta_{x} \in \gamma'$ be any point following $x'$ such that $d(x',\beta_{x}) \geq 4d(x',x)$ [if these choices are not possible because $\gamma'$ terminates, then set $\alpha_{x}(\beta_{x})$ to be equal to the terminal point of $\gamma'$ which precedes (follows) $x'$].  

Since $\gamma'$ is an $M'$--Morse quasi-geodesic and because $[\alpha_{x},\beta_{x}]$ is a geodesic with endpoints on $\gamma',$ it follows that $[\alpha_{x},\beta_{x}]$ is similarly $M''(K,L)$--Morse, where the constant $M''(K,L)=M'(0,1)+M'(K,L).$  In particular, the constant $M''$ only depends on $M'.$  Then, by assumption the geodesic $[\alpha_{x},\beta_{x}]$ is $C(M'')=C(M')$--strongly contracting.  By Lemma \ref{lem:reduce} it follows that $\gamma'|_{[\alpha_{x},\beta_{x}]}$ is $C'(C(M'),M')=C'(M')$--strongly contracting.  In particular, since for all $y\in X$ such that $d(x,y) \leq d(x,x'),$ we know that $\pi_{\gamma'}(y) \subset \gamma'|_{[\alpha_{x},\beta_{x}]},$  it follows that $d(\pi_{\gamma'}(x),\pi_{\gamma'}(y))\leq C'(M').$  Since for any starting $x \in X$ we can preform this process of creating such an interval $[\alpha_{x},\beta_{x}]$ and proceeding as above, it follows that the quasi-geodesic $\gamma'$ is $C'(M')$--strongly contracting.
\end{proof}

We are now prepared to prove the main theorem.
\begin{thm} \label{thm:main}
Let $X$ be a CAT(0) space and $\gamma \subset X$ a (K,L)--quasi-geodesic.  Then the following are equivalent:
\begin{enumerate}
\item $\gamma$ is $C'$--strongly contracting, 
\item $\gamma$ is (b,c)--contracting,
\item $\gamma$ is $M$--Morse, and
\item $\gamma$ is $S$--slim.
\end{enumerate}
Moreover, any one of the four sets of constants $\{(b,c),C',M,S\}$ can be written purely in terms of the any of the others in conjunction with the quasi-isometry constant (K,L).  
\end{thm}
\begin{proof}
By definition (1)$\implies$(2).  The fact that (2)$\implies$(3) is a slight generalization of the well known ``Morse stability lemma.''  For an explicit proof see Lemma 3.3 in \cite{sultanmorse} (or similarly Lemma 5.13 in \cite{algomkfir}).  In order to complete the proof of the theorem we will provide an explicit proof that: (3)$\implies$(2), (3) [+(2)] $\implies$(1), (1)+(3)$\implies$(4), and (4)$\implies$(2).   

\textbf{(3)$\implies$(2)}:  By Corollary \ref{cor:reduce} it suffice to prove (3)$\implies$(2) in the special case of $\gamma$ a geodesic.  Fix $x,y \in X$ such that $d(x,y) \leq \frac{1}{4}d(x, \pi_{\gamma}(x)).$  Set $A=d(x, \pi_{\gamma}(x)),$ $D=d( \pi_{\gamma}(x), \pi_{\gamma}(y)).$  Note that  $\frac{3A}{4} \leq d(y, \pi_{\gamma}(y))\leq \frac{5A}{4}.$  

Let $\rho_{1}:[0,1]\ra X$ be the geodesic parameterized proportional to arc length joining $ \pi_{\gamma}(x)=\rho_{1}(0)$ and $x=\rho_{1}(1).$  Similarly, let $\rho_{2}:[0,1]\ra X$ be the geodesic parameterized proportional to arc length joining $ \pi_{\gamma}(y)=\rho_{2}(0)$ and $y=\rho_{2}(1).$  Note that by property [C1] of Lemma \ref{lem:cat} $\frac{A}{4} \geq D.$  Set $s=\frac{D}{A},$ so $s \in [0,1].$  Applying property [C2] of Lemma \ref{lem:cat} to the geodesics $\rho_{1},\rho_{2},$ we have that 
\begin{eqnarray*}
d(\rho_{1}(s),\rho_{2}(s)) &\leq&  (1-s) d( \pi_{\gamma}(x), \pi_{\gamma}(y)) + sd(x,y) \\
     &\leq&  (1-s)D + s\frac{A}{4} \leq D+ \frac{D}{A}\frac{A}{4} = \frac{5D}{4}   \\
\end{eqnarray*}

As in the discussion proceeding Lemma \ref{lem:genquasi} with $\rho_{1}(s)$ taking the place of $x$ and $\rho_{2}(s)$ taking the place of $y,$ we can construct a quasi-geodesic $\phi$ composed of the concatenation of five geodesic segments.  By construction, in this case our constants are $a=1,b \in [1,\frac{5}{4}],$ and $c \in [\frac{3}{4},\frac{5}{4}].$  By Lemma \ref{lem:genquasi} $\phi$ is a (15,0)--quasi-geodesic.  Furthermore, since $a+c-b \geq \frac{1}{2},$ again by Lemma \ref{lem:genquasi} it follows that $D \leq 4M(15,0).$  Hence, we have just shown that if $\gamma$ is M(K,L)--Morse then it is $(\frac{1}{4},4M(15,0))$--contracting.  This completes the proof of (3)$\implies$(2).  

\textbf{(3) [+(2)]$\implies$(1)}:  By Corollary \ref{cor:reduce} it suffice to prove (3)$\implies$(1) in the special case of $\gamma$ a geodesic.  

Fix $x,y \in X$ such that $d(x,y) < d(x,\pi_{\gamma}(x)).$  Set $A=d(x, \pi_{\gamma}(x)),$ $D=d( \pi_{\gamma}(x), \pi_{\gamma}(y)),$ and $B=d( \pi_{\gamma}(y),y).$   In order to complete the proof we must bound $D.$  

Without loss of generality we can assume that $B < \frac{A}{100}.$  If not, then by the previous step of (3)$\implies$(2), in conjunction with Lemma \ref{lem:close}, it follows that $\gamma$ is both $(\frac{99}{100},K_{99/100})$--contracting and $(\frac{9901}{10,000},K_{9901/10,000})$--contracting, where the constants $K_{99/100}, K_{9901/10,000}$ depend only on the Morse constant.  Let $z \in [x,y]$ such that $d(z,y) = \frac{99A}{10,000}$ [if this is not possible, namely $d(x,y) < \frac{99A}{10,000},$ then set $z=y$].  But then, $d(x,z)= d(x,y)-d(y,z)\leq A -  \frac{99A}{10,000} = A\frac{9901}{10,000}.$  It follows that $$D = d(\pi_{\gamma}(x),\pi_{\gamma}(y)) \leq d(\pi_{\gamma}(x),\pi_{\gamma}(z)) + d(\pi_{\gamma}(z),\pi_{\gamma}(y)) \leq K_{9901/10,000} + K_{99/100} .$$

Similarly, without loss of generality we can assume $A\geq 3D.$  If not, then by the previous step of (3)$\implies$(2), in conjunction with Lemma \ref{lem:close}, it follows that $\gamma$ is $(\frac{3}{4},K_{3/4})$--contracting, where the constant $K_{3/4}$ depends only on the Morse constant.  Let $z \in [x,y]$ such that $d(x,z) = \frac{3A}{4}$ [if this is not possible, namely $d(x,y) < \frac{3A}{4},$ then set $z=y$].  But then, 
\begin{eqnarray*} 
D &=& d(\pi_{\gamma}(x),\pi_{\gamma}(y)) \leq d(\pi_{\gamma}(x),\pi_{\gamma}(z)) + d(\pi_{\gamma}(z),\pi_{\gamma}(y)) \\
&\leq& K_{3/4} +d(z,y) \leq K_{3/4} +\frac{A}{4} \leq K_{3/4} +\frac{3D}{4} \\
&\implies& D \leq 4K_{3/4}.
\end{eqnarray*}

We complete the proof of this step by considering two cases:

\indent \textbf{Case 1: $B \geq 2D$}

Consider the triangle $\triangle=\triangle(x,y,\pi_{\gamma}(y)).$  Let the triangle $\overline{\triangle}= \overline{\triangle}(\overline{x},\overline{y},\overline{\pi_{\gamma}(y)})$ be the corresponding comparison triangle in Euclidean space.  Since $|[x, \pi_{\gamma}(y)]| \geq |[x, \pi_{\gamma}(x)]|=A,$ while $|[x,y]| < A,$ and $|[y, \pi_{\gamma}(y)]| < \frac{A}{100} <A,$ it follows that the angle in the comparison triangle $\overline{\triangle}$ between the sides $[\overline{x},\overline{\pi_{\gamma}(y)}]$ and $[\overline{y},\overline{ \pi_{\gamma}(y)}])$ is less than $90^{o}.$  Let $\overline{p} \in [\overline{x},\overline{\pi_{\gamma}(y)}]$ such that $d(\overline{p},\overline{\pi_{\gamma}(y)})=2D.$  Similarly, let $\overline{q} \in [\overline{y},\overline{\pi_{\gamma}(y)}]$ such that $d(\overline{q},\overline{\pi_{\gamma}(y)})=2D.$  Since the angle between the sides $[\overline{p},\overline{\pi_{\gamma}(y)}]$ and $[\overline{q},\overline{ \pi_{\gamma}(y)}])$ is less than $90^{o},$ it follows by elementary Euclidean geometry that $d(\overline{p},\overline{q}) \leq 2\sqrt{2}D.$  Hence by the CAT(0) comparison property  $d(p,q) \leq 2\sqrt{2}D.$   

\begin{figure}
  \centering
  \subfloat{\label{fig:case1}\includegraphics[width=0.3\textwidth]{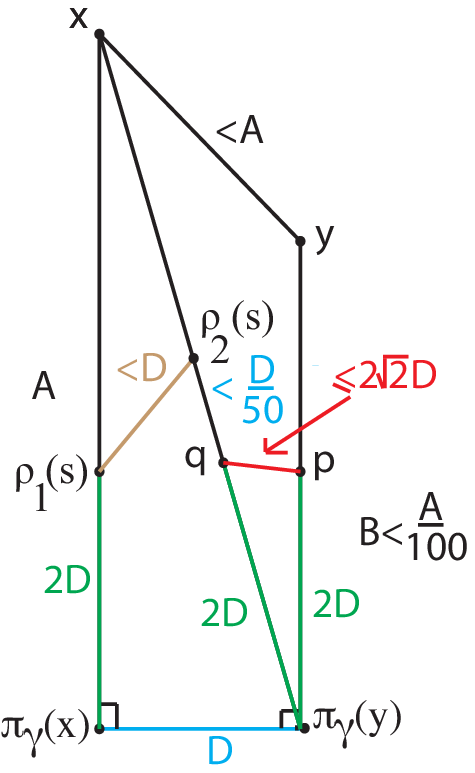}}
  ~ 
  \subfloat{\label{fig:case2}\includegraphics[width=0.28\textwidth]{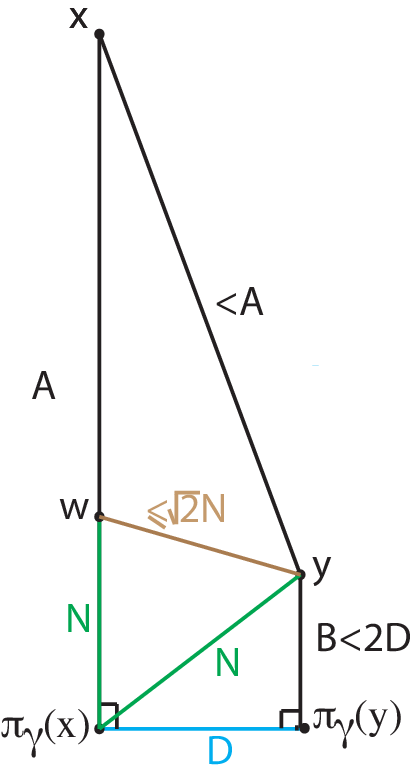}}
  ~ 
  \caption{Illustration of Cases 1 and 2 in the proof of (3)$\implies$(1)}
  \label{fig:cases3implies1}
\end{figure}

Let $\rho_{1}:[0,1]\ra X$ be the geodesic parameterized proportional to arc length joining $ \pi_{\gamma}(x)=\rho_{1}(0)$ and $x=\rho_{1}(1).$  Similarly, let $\rho_{2}:[0,1]\ra X$ be the geodesic parameterized proportional to arc length joining $ \pi_{\gamma}(y)=\rho_{2}(0)$ and $x=\rho_{2}(1).$  Note that by our assumptions, $|\rho_{1}| \leq |\rho_{2}| \leq \frac{101}{100} |[\rho_{1}]|.$     

Set $s=\frac{2D}{A},$ so by our assumptions $s \in [0,1].$  Applying property [C2] of Lemma \ref{lem:cat} to the geodesics $\rho_{1},\rho_{2},$ we have that 
\begin{eqnarray*}
d(\rho_{1}(s),\rho_{2}(s)) &\leq&  (1-s) d( \pi_{\gamma}(x), \pi_{\gamma}(y)) + sd(x,x) \\
     &\leq&  (1-s)D + 0  = D   \\
\end{eqnarray*}

Recall that we let $p \in [x,\pi_{\gamma}(y)]$ be such that $d(p,\pi_{\gamma}(y))=2D.$  Similarly, we let $q \in [y,\pi_{\gamma}(y)]$ be such that $d(q,\pi_{\gamma}(y))=2D.$  By definition $\rho_{1}(s) \in [x,\pi_{\gamma}(x)]$ satisfies $d(\rho_{1}(s),\pi_{\gamma}(x))=2D.$  Furthermore, since $|\rho_{1}| \leq |\rho_{2}| \leq \frac{101}{100} |[\rho_{1}]|,$ it follows that $d(\rho_{2}(s),p) \leq \frac{D}{50}.$  Putting things together,
\begin{eqnarray*}
d(\rho_{1}(s),q) &\leq& d(\rho_{1}(s),\rho_{2}(s)) + d(\rho_{2}(s), p ) + d(p,q) \\
 &\leq& D + \frac{D}{50} + 2\sqrt{2}D \leq 3.9D \\
\end{eqnarray*}

As in the discussion proceeding Lemma \ref{lem:genquasi} with $\rho_{1}(s)$ taking the place of $x$ and $q$ taking the place of $y,$  we can construct a quasi-geodesic $\phi$ composed of the concatenation of five geodesic segments.  By construction, in this case our constants are $a,c=2,b \leq 3.9.$  By Lemma \ref{lem:genquasi} $\phi$ is a (33,0)--quasi-geodesic.  Furthermore, since $a+c-b \geq 0.1,$ again by Lemma \ref{lem:genquasi} it follows that $D \leq 200M(33,0).$  This completes the proof of (3)$\implies$(1) in this case. 

\indent \textbf{Case 2: $B \leq 2D$} 

Let $N= d(\pi_{\gamma}(x),y).$  In this case, by the triangle inequality $$N \leq d(\pi_{\gamma}(x),\pi_{\gamma}(y)) + d(\pi_{\gamma}(y),y) \leq 3D.$$   

Consider the triangle $\triangle=\triangle(x,y,\pi_{\gamma}(x)).$  Let the triangle $\overline{\triangle}= \overline{\triangle}(\overline{x},\overline{y},\overline{\pi_{\gamma}(x)})$ be the corresponding comparison triangle in Euclidean space.  Since $|[x, \pi_{\gamma}(x)]| =A,$ while $|[x,y]| < A,$ and $|[y, \pi_{\gamma}(x)]| \leq 3D \leq A,$ it follows that the angle in the comparison triangle $\overline{\triangle}$ between the sides $[\overline{x},\overline{\pi_{\gamma}(x)}]$ and $[\overline{y},\overline{ \pi_{\gamma}(x)}])$ is less than $90^{o}.$  Let $\overline{w} \in [\overline{x},\overline{\pi_{\gamma}(x)}]$ such that $d(\overline{w},\overline{\pi_{\gamma}(x)})=N.$  Since the angle between the sides $[\overline{w},\overline{\pi_{\gamma}(x)}]$ and $[\overline{y},\overline{ \pi_{\gamma}(x)}])$ is less than $90^{o},$ it follows by elementary Euclidean geometry that $d(\overline{w},\overline{y}) \leq \sqrt{2}N.$  Hence by the CAT(0) comparison property  $d(w,y) \leq \sqrt{2}N.$

To complete the proof of (3)$\implies$(1), we will consider two subcases:

\textbf{Case 2a: $B\geq \frac{N}{2}$} 

As in the discussion proceeding Lemma \ref{lem:genquasi} with $w$ taking the place of $x$ and $y$ standing in for itself,  we can construct a quasi-geodesic $\phi$ composed of the concatenation of five geodesic segments. By construction, in this case our constants are $a \leq 3 ,c\leq2,b \leq \sqrt{2}.$  By Lemma \ref{lem:genquasi} $\phi$ is a (27,0)--quasi-geodesic.  Furthermore, in this subcase our assumptions ensure that  $a+c-b \geq 1.5 - \sqrt{2},$ and hence again by Lemma \ref{lem:genquasi} it follows that $D \leq \frac{2M(27,0)}{1.5-\sqrt{2}}.$  This completes the proof of (3)$\implies$(1) in this subcase. 

\textbf{Case 2b: $B\leq \frac{N}{2}$} 

By the triangle inequality $N \leq D+B.$  Hence, by the assumption of the subcase, $D \geq \frac{N}{2}.$  On the other hand, let $z \in [w,y]$ such that $d(w,z)=\frac{99N}{100}$ [If this is not possible, namely $d(w,y)< \frac{99N}{100},$ set $z=y].$  Since by the previous step of (3)$\implies$(2), in conjunction with Lemma \ref{lem:close}, it follows that $\gamma$ is $(\frac{99}{100},K_{99/100})$--contracting, where the constant $K_{99/100}$ depends only on the Morse constant.  Then,

\begin{eqnarray*} 
\frac{N}{2} &\leq& D = d(\pi_{\gamma}(x),\pi_{\gamma}(y)) \leq d(\pi_{\gamma}(x),\pi_{\gamma}(z)) + d(\pi_{\gamma}(z),\pi_{\gamma}(y)) \\
&\leq& K_{99/100} +d(z,y) \leq K_{99/100} +N(\sqrt{2} - \frac{99}{100}) \\
&\implies& N \leq \frac{K_{99/100}}{\frac{1}{2}+\frac{99}{100} -\sqrt{2}} 
\end{eqnarray*}

However, since property [C1] of Lemma \ref{lem:cat} ensures that $D \leq N,$ it follows that in this subcase $D \leq \frac{K_{99/100}}{\frac{1}{2}+\frac{99}{100} -\sqrt{2}},$ thus completing the proof in this subcase and hence the proof of (3)$\implies$(1).

\begin{figure}[htpb] 
\centering
\includegraphics[height=5 cm]{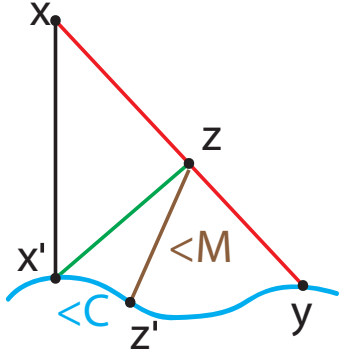}
\caption{(1)+(3)$\implies$(4).}\label{fig:impliesslim}
\end{figure}

\textbf{(1)+(3)$\implies$(4)}:  Fix $x \in X,$ $x' \in \pi_{\gamma}(x),$ and $y \in \gamma.$  Let $z=\pi_{[x,y]}(x'),$ and let $z' \in \pi_{\gamma}(z).$  Since $\pi_{[x,y]}(x)=x,$ and $\pi_{[x,y]}(x')=z,$ by property [C1] of Lemma \ref{lem:cat} it follows that $d(x,z)\leq d(x',x).$  By $C'$--contraction of $\gamma,$ it follows that $d(x',z')<C'.$  Furthermore, by Lemma \ref{lem:quasigeodesic}, the concatenated path $[x',z] \cup [z,y]$ is a (3,0)--quasi-geodesic.  In particular, it follows that $d(z',z)$ is bounded above by the Morse constant $M(3,0).$  Hence, $d(x',[x,y])=d(x',z) \leq d(x',z')+d(z',z) \leq C' +M(3,0).$  Thus, $\gamma$ is $(C'+M(3,0))$-slim, thus completing this step of the proof.  See Figure \ref{fig:impliesslim}.

\textbf{(4)$\implies$(2)}: Assume $\gamma$ is an S-slim, (K,L)--quasi-geodesic.  Fix $x \in X,$ $x' \in \pi_{\gamma}(x).$  Let $y \in X$ be any point such that  $d(x,y)\leq \frac{d(x,x')}{2K},$ and fix any $y' \in \pi_{\gamma}(y).$  We will show that $d(x',y')$ is bounded above by the constant $8S+6L$ thus showing that $\gamma$ is $(\frac{1}{2K}, 8S+6L)$-contracting.  

\begin{figure}[htpb] 
\centering
\includegraphics[height=5 cm]{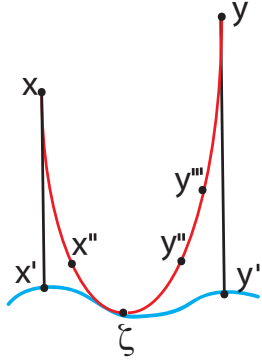}
\caption{(4) $\implies$(2).}\label{fig:slimimplies}
\end{figure}

Consider the function $f:X \ra \R$ defined by $f(a)=d(x',a) - d(y',a).$  Restricting the function $f$ to $\gamma|_{[x',y']}$ the function can have jump discontinuities of at most $2L,$ and hence by the intermediate value theorem, $\exists \zeta \in \gamma$ such that $d(y',\zeta)-L  \leq d(x',\zeta) \leq  d(y',\zeta)+L.$  See Figure \ref{fig:slimimplies}.  Without loss of generality we can assume $d(x',\zeta) \geq 2(S+L),$ for if not, then $d(x',y')\leq 4S+6L \leq 8S+6L,$ in which case we are done.  Furthermore since $\gamma$ is a (K,L)--quasi-geodesic, it follows that 
\begin{eqnarray}
2d(\zeta,x')-L &\leq& d(\zeta,x') + d(\zeta,y') \leq d_{\gamma}(x',y') \leq Kd(x',y') + L \\
\label{eq:recall} &\implies& d(\zeta,x') \leq \frac{Kd(x',y')}{2} +L
\end{eqnarray}  

Let $x''\in [x,\zeta]$ such that $d(\zeta,x'')=d(\zeta,x').$  Similarly, let $y''\in [y,\zeta]$ such that $d(\zeta,y'')=d(\zeta,y').$  By the remarks following Definition \ref{def:slim} of S--slim, $d(x',x''), d(y',y'')\leq 2S.$  

Let $y''' \in [\zeta,y]$ be such that $d(y''',\zeta)=\frac{d(\zeta,x'')d(y,\zeta)}{d(x,\zeta)}.$  Comparing the lengths of $[x,\zeta]$ and $[y,\zeta],$ we have the following inequality:
$$\frac{2k-1}{2k} |[x,\zeta]|   \leq  |[x,\zeta]| - \frac{ |[x,x']|}{2k} \leq |[y,\zeta]| \leq |[x,\zeta]| + \frac{ |[x,x']|}{2k} \leq  \frac{2k+1}{2k} |[x,\zeta]|.$$
In particular, since $\frac{2k-1}{2k} |[x,\zeta]|   \leq |[y,\zeta]| \leq  \frac{2k+1}{2k} |[x,\zeta]|,$ it follows that $d(y'',y''')\leq \frac{d(\zeta,x'')}{2K} +2L.$

Applying CAT(0) thinness of triangles to the triangle $\triangle(x,y,\zeta),$ it follows that $$|[x'',y''']| \leq \frac{d(\zeta,x'')}{d(x,\zeta)}\frac{d(x,x')}{2K} \leq \frac{d(\zeta,x'')}{2K}.$$  Putting things together the following inequality completes the proof:
\begin{eqnarray*}
d(x',y') &\leq& d(x',x'') + d(x'',y''')+d(y''',y'') + d(y'',y') \\
&\leq& 2S + \frac{d(\zeta,x'')}{2K} + \left(\frac{d(\zeta,x'')}{2K} +2L \right)+ 2S = 4S+ 2L + \frac{d(\zeta,x'')}{K} \\
&\leq& 4S+ 2L + (\frac{d(x',y')}{2} + \frac{L}{K}) \;\;\; [\mbox{ by  Equation \ref{eq:recall} }]\\
&\implies& d(x',y') \leq 8S +4L +\frac{2L}{K} \leq 8S + 6L.
\end{eqnarray*}        
%

\end{proof}

Notice that of the four equivalent definitions of hyperbolic type quasi-geodesics considered in Theorem \ref{thm:main}, the Morse version is particular well suited with regard to quasi-isometries.  In particular, let $\gamma$ be a $M$--Morse quasi-geodesic.  Then for $f:X \ra Y$ a (K,L)--quasi-isometry, by definition $f(\gamma)$ is an $M'(K,L)$--Morse quasi-geodesic.  In light of Theorem \ref{thm:main} we immediately obtain the following corollary, which as noted in the introduction has application in \cite{charney}:
\begin{cor} \label{cor:quasiisometry}.
Let $X$ be a CAT(0) space, $\gamma \subset X$ a C--strongly contracting (K',L')--quasi-geodesic, and $f:X \ra X$ a $(K,L)$ quasi-isometry.  Then $f(\gamma)$ is $C'(C,K,L,K',L')$--strongly contracting quasi-geodesic. 
\end{cor}


\bibliographystyle{amsalpha}

\begin{thebibliography}{[BMM]} 
\bibitem[A]{algomkfir} Y. Algom-Kfir.  Strongly contracting geodesics in Outer Space.  Preprint \href{http://arxiv.org/abs/0812.1555}{arXiv:0812.155} (2010).
\bibitem[Bal1]{bal1} W. Ballman.  Nonpositivelycurved manifolds of higher rank. \emph{Ann. of Math.} \textbf{122} (1985) 597--609.
\bibitem[Bal2]{bal2} W. Ballman.  Lectures on spaces of nonpositive curvature.   DMV Seminar, vol 25 Birkhauser, 1995.
\bibitem[Bal3]{bal3} W. Ballman.  Lectures on spaces of nonpositive curvature.   DMV Seminar, vol 25 Birkhauser Verlag, Basel, 1995, with appendix by Misha Brin.
\bibitem[B]{behrstock}  J. Behrstock.  Asymptotic geometry of the mapping class group and Teichm\"uller space.  \emph{Geom. Topol.} \textbf{10} (2006) 1523--1578.
\bibitem[BC]{behrstockcharney}J. Behrstock and R. Charney, Divergence and quasimorphisms of right-angled Artin groups, \emph{Math. Ann.} (2011), 1--18.
\bibitem[BD]{behrstockdrutu} J. Behrstock and C. Drutu, Divergence, thick groups, and short conjugators,  Preprint \href{http://arxiv.org/abs/1110.5005}{1110.5005} (2011). 
\bibitem[BeF]{bestvina} M. Bestvina and K. Fujiwara, A characterization of higher rank symmetric spaces via
bounded cohomology, \emph{Geom. Funct. Anal.} \emph{19} (2009) 11--40.
\bibitem[BH]{bridson}  M. Bridson and A. Haefliger, Metric Spaces of Non-positive Curvature, Grad. Texts in Math. 319, Springer-Verlag, New York, (1999).
\bibitem[BrF]{brockfarb}J. Brock and B. Farb, Curvature and rank of Teichm\"uller space,  \emph{Amer. J. Math.} \textbf{128} (2006), 1-22.
\bibitem[BrM] {brockmasur} J. Brock, H. Masur. Coarse and Synthetic Weil-Petersson Geometry: quasißats, geodesics, and relative hyperbolicity. \emph{Geom. Topol.} \textbf{12} (2008) 2453--2495.
\bibitem[BMM] {bmm}J. Brock, H. Masur, and Y. Minsky, Asymptotics of Weil-Petersson geodesics II: bounded geometry and unbounded entropy,  \emph{Geom. Funct. Anal.} \textbf{21} (2011), 820--850.
\bibitem[Cha]{charney}R. Charney.  (2011) lecture entitled ``Contracting boundaries of CAT(0) spaces,'' at \emph{3-manifolds, Artin Groups, and Cubical Geometry CBMS-NSF conference}, New York, NY CUNY. 
\bibitem[DMS]{drutumozessapir} C. Drutu, S Mozes, and M. Sapir.  Divergence in lattices in semisimple Lie groups and graphs of groups.  \emph{Trans. Amer. Math. Soc.} \textbf{362} (2010), 2451--2505.  
\bibitem[DS]{drutusapir} C. Drutu and M. Sapir, Tree-graded spaces and asymptotic cones of groups, \emph{Topology} \textbf{44} (2005), 959--1058, With an appendix by Denis Osin and Mark Sapir.
\bibitem[KL]{kapovitchleeb} M. Kapovich and B. Leeb.  3-Manifold groups and nonpositive curvature.  \emph{Geom. Funct. Anal.} \textbf{8} (1998), 841--852.
\bibitem[MM]{mm1}H. Masur and Y. Minsky, Geometry of the complex of curves. I. hyperbolicity, \emph{Invent. Math.} \textbf{138} (1999) 103--149.
\bibitem[Mor]{morse}H. Morse, A fundamental class of geodesics on any closed surface of genus greater than one, \emph{Trans. AMS} \textbf{26} (1924) 25--60.
\bibitem[Mos]{mosher} L. Mosher, Stable Teichm\"uller quasigeodesics and ending laminations, \emph{Geom. Topol.} \textbf{7} (2003) 33--90.
\bibitem[Osi]{osin} D. Osin.  Relatively hyperbolic groups: intrinsic geometry, algebraic properties, and algorithmic problems.  \emph{Mem. amer. Math. Soc.} \textbf{179} no. 843 (2006) vi+100pp.
\bibitem[Sul]{sultanmorse}H. Sultan, Hyperbolic quasi-geodesics in CAT(0) spaces, preprint arXiv:1112.4246, submitted, 2011.
\end{thebibliography}

\end{document}